\documentclass[12pt, reqno]{amsart}
\usepackage{amssymb}
\usepackage{amsfonts}
\usepackage{amssymb,amsmath,amscd}
\usepackage[colorlinks=true]{hyperref}
\usepackage{mathrsfs}
\newtheorem{Theorem}{\indent Theorem}[section]

\newtheorem{Lemma}[Theorem]{\indent Lemma}
\newtheorem{Corollary}[Theorem]{\indent Corollary}

\newtheorem{Proposition}[Theorem]{\indent Proposition}

\theoremstyle{remark}

\begin{document}
\centerline{\bf On sums of $k$-th powers with almost equal primes
}
\bigskip
\centerline{\small  Wei Zhang}
\bigskip

\textbf{Abstract}\
 For ``almost all" sufficiently large $N,$ satisfying necessary congruence conditions and $k\geq 2$,  we show that
  there is an {\bf asymptotic formula} for the number of solutions of the equation
\begin{align*}
\begin{split}
&N=p_{1}^{k}+p_{2}^{k}+\cdots+p_{s}^{k},
\\
&\left|p_{i}-( N/s)^{1/k}\right|\leq (N/s)^{\theta/k},\  (1\leq i\leq s)
\end{split}
\end{align*}
with
\begin{align*}
s\geq \frac{k(k+1)}{2}+1\ \ \textup{and}\ \
\theta\geq
  {\bf 2/3}+\varepsilon.
\end{align*}

This  enlarges the effective range of $s$ for which can be obtained by the method of M\"{a}tomaki and Xuancheng Shao \cite{MS}.
[Discorrelation  between primes in short intervals and polynomial phase, Int. Math. Res. Not. IMRN 2021, no. 16, 12330-12355.]

The idea is to avoid using the exponential sums (\ref{MS0}) and Vinogradov mean value theorems in Lemma \ref{K} simultaneously. And the main new ingredient is from Kumchev and Liu \cite{KL} (see Lemma \ref{t2}).

\medskip


\bigskip
\bigskip
\numberwithin{equation}{section}

\section{Introduction}
For $k\in \mathbb{Z}^{+}$ and a prime $p,$ let $\tau=\tau(k,p)$ be the largest integer such that $p^{\tau}|k.$ Define
\[
\eta(k,p)=\tau+2
\]
when $p=2$ and $\tau>0,$ and $\eta(k,p)=\tau+1$ otherwise. Then define
\[
R=R(k)=\prod_{(p-1)|k}p^{\eta}.
\]
 For sufficiently large $N,$ $N\in \mathcal{N}_{s}=\{N\in \mathbb{N}: N\equiv s ( \mathrm{mod}\ R)  \ \textup{and}\ 9\nmid N,\ \textup{for}\ k=3\ \textup{and}\ s=7 \}$,
\begin{align}\label{1.1}
\begin{split}
&N=p_{1}^{k}+p_{k}^{k}+\cdots+p_{s}^{k},
\\
&\left|p_{i}-( N/s)^{1/k}\right|\leq (N/s)^{\theta/k},\  (1\leq i\leq s)
\end{split}
\end{align}
was investigated by many experts. For example, see \cite{K2,KL,LLZ,MS,Sa,WW}.
In \cite{K2}, this was first studied for $s\geq7$ and $k=2$. Recently, this was improved by many authors. For example, see \cite{KL,MS,Sa,WW}.

In \cite{KL}, Kumchev and Liu proved that for sufficiently large $N,$ $N\in \mathcal{N}_{s}=\{N\in \mathbb{N}: N\equiv s ( \mathrm{mod}\ R)\},$ there are solutions of the equation (\ref{1.1})
with $s\geq k^{2}+k+1$ and $\theta\geq31/40+\varepsilon,$ which gives an asymptotic lower bound of the correct order for the number of the representations.
They also proved that ``almost all" sufficiently large $N,$ $N\in \mathcal{N}_{s}=\{N\in \mathbb{N}: N\equiv s ( \mathrm{mod}\ R)\}$, (\ref{1.1})
is solvable  with $s\geq \frac{k(k+1)}{2}+{\bf 1}$ and $\theta\geq31/40+\varepsilon.$
Next we will introduce the definition of
``almost all".
For primes $p_{i}$, define
\[E(N,s,x^{k-1}y)=|\mathcal{E}(N,s,x^{k-1}y)|,\]
where
\begin{align*}
&\mathcal{E}(N,s,x^{k-1}y)\\
&=
\left\{n:n\in(N,N+x^{k-1}y]\cap\mathcal{N}_{s},\ \varrho(n)-\mathfrak{S}(n)\mathfrak{J}(n)\geq y^{s-1}x^{1-k}(\log x)^{-1}\right\},
\end{align*}
where $\varrho(n)$ is the number of representations of $n$ in the form of (\ref{1.1}) and $\mathfrak{S}(n), \mathfrak{J}(n)$ are defined by ($\ref{asy}).$
If
\[
\lim_{N\longrightarrow\infty}\frac{E(N,s,x^{k-1}y)}{x^{k-1}y}=0
\ \ \left(\textup{i.e.}\ E(N,s,x^{k-1}y)=o\left(x^{k-1}y\right)\right),\]
then we say that for ``almost all" sufficiently large $N,$ $N\in \mathcal{N}_{s},$ there is an {\bf asymptotic formula} for the number of solutions of the equation (\ref{1.1})
with certain $s$ and $\theta.$

Recently, these results were  improved  by Matom\"aki-Shao \cite{MS} and  Salmensuu \cite{Sa} respectively.
Matom\"aki-Shao \cite{MS} use the exponential sums. In order to use the exponential sums
\begin{align}\label{MS0}
\sum_{x<n\leq x+x^{2/3+\varepsilon}}\Lambda(n)e(\alpha n^{k}),
\end{align}
they need more variables to apply a better type Vinogradov mean value theorem related to \cite{D} instead of Lemma \ref{K}.
 By using the Harman sieve, Salmensuu  got some even better results.
Unfortunately, both their results do not contain the situation of ``almost all" for $s\leq k^{2}+k$. However,  the method of Matom\"aki-Shao \cite{MS} implies that
``almost all" sufficiently large $N,$ $N\in \mathcal{N}_{s}=\{N\in \mathbb{N}: N\equiv s ( \mathrm{mod}\ R)\},$
there is an {\bf asymptotic formula} for the number of solutions of the equation (\ref{1.1})
with $s\geq (k^{2}+k)/2+2$ and $\theta\geq2/3+\varepsilon.$
This improves the previous results for the aspect of $\theta$ but with worse $s$ (as cited in \cite{MS}).
The aim of this paper is to reduce the $\frac{k(k+1)}{2}+{\bf2}$ to $\frac{k(k+1)}{2}+{\bf1}.$ Then the range of $s$ is the same as previous results in \cite{WW,KL}.
The idea is to avoid using the exponential sums (\ref{MS0}) and the Vinogradov mean value theorems in Lemma \ref{K} simultaneously.

\begin{Theorem}\label{t3}
For ``almost all" sufficiently large $N,$ $N\in \mathcal{N}_{s}=\{N\in \mathbb{N}: N\equiv s ( \mathrm{mod}\ R)\},$ there is an {\bf asymptotic formula} for the number of solutions of the equation
\begin{align*}
\begin{split}
&N=p_{1}^{k}+p_{2}^{k}+\cdots+p_{s}^{k},
\\
&\left|p_{i}-( N/s)^{1/k}\right|\leq (N/s)^{\theta/k},\  (1\leq i\leq s)
\end{split}
\end{align*}
with
\[
s\geq\frac{k(k+1)}{2}+1\ \ \textup{and}\ \  \ \theta\geq2/3+\varepsilon.
\]
\end{Theorem}

In fact, such an idea can also be used to  enlarge the effective range of $s$ of \cite{MS}.
Though there are some other results in \cite{Sa} with much better $\theta.$ However, such type results cannot be used to give something related to  {\bf asymptotic formulas} for the number of solutions of the equation (\ref{1.1}). By similar arguments, one can get the following corollary.

\begin{Corollary}\label{t4}
For all sufficiently large $N,$ $N\in \mathcal{N}_{s}=\{N\in \mathbb{N}: N\equiv s ( \mathrm{mod}\ R)\},$ there is an {\bf asymptotic formula} for the number of solutions of the equation
\begin{align*}
\begin{split}
&N=p_{1}^{k}+p_{2}^{k}+\cdots+p_{s}^{k},
\\
&\left|p_{i}-( N/s)^{1/k}\right|\leq (N/s)^{\theta/k},\  (1\leq i\leq s)
\end{split}
\end{align*}
with
\[
s\geq k(k+1)+1\ \ \textup{and}\ \  \ \theta\geq2/3+\varepsilon.
\]
\end{Corollary}

\section{Preliminaries}

For $v\in\mathbb{Z}^{+},$ let $w_{k}(q)$ be defined as
\begin{align}
w_{k}(p^{ku+v})=
\begin{cases}
kp^{-u-1/2}    &\textup{for}\ \  u\geq 0 \ \textup{and}\ \ v=1,\\
p^{-u-1}       &\textup{for}\ \  u\geq 0 \ \textup{and}\ \ 2\leq v\leq k.
\end{cases}
\end{align}
Then we will introduce two lemmas. These lemmas are from \cite{Yao1} (in Chinese) and \cite{KL} respectively.
\begin{Lemma}\label{Y}
Let $\log D\ll \log x$ and $1\leq\max\{x^{1/2}, M\}\leq y\leq x.$ Then there exist a constant $c_{0}>0$ uniformly for $\gamma\in \mathbb{R}$ such that
\[
\sum_{q\leq M}
\sum_{\mbox{\tiny$\begin{array}{c}
1\leq a\leq q\\
(a,q)=1\\
\end{array}$}}\int_{|\alpha-a/q|\leq1}\frac{w_{k}(q)^{2}|\sum_{x<p\leq x+y}(\log p)e(p^{k}(\alpha+\gamma))|^{2} }{1+D|\alpha-a/q|}d\alpha\ll y^{2}D^{-1}\log^{c_{0}}x.
\]
\end{Lemma}

\begin{proof}
One can refer to the case $k=3$ and $k=4$ in \cite{Yao} and \cite{YZ}, respectively.
And this is Lemma 2.1 in \cite{Yao1} (in Chinese). So we only sketch the proof.
For $\alpha=a/q+\gamma_{1},$ for the properties of geometric series, we have
\[
\sum_{1\leq a\leq q}\left|\sum_{x<p\leq x+y}(\log p)e(p^{k}(\alpha+\gamma))\right|^{2}\leq
q \sum_{x <p_{1},\ p_{2} \leq x+y \atop p_{1}^{k}\equiv p_{2}^{k}(\textup{mod}q)}(\log p_{1})(\log p_{2}).
\]
On the other hand, we have the estimate
\[
 q\sum_{x <p_{1},\ p_{2} \leq x+y,\ p_{1}^{k}\equiv p_{2}^{k}(\textup{mod}\ q)}1
 \ll y^{2}d(q)^{c_{k,1}}.
\]
As (see Lemma 2.1 in \cite{Zh})
\[
\sum_{q\leq M}w_{k}(q)^{2}d(q)^{c_{k,2}}\ll (\log M)^{c_{k,3}}.
\]
Then the desired conclusion follows. For similar proofs,  one can refer to  \cite{B1,B2,Yao,YZ,Zh}.
\end{proof}

Next Lemma is from \cite{KL}.
\begin{Lemma}\label{t2}
Suppose that $y=x^{\theta},$ $0<\rho\leq t_{k}^{-1},$ $\frac{1}{2-t_{k}\rho}\leq \theta\leq 1$ and $A:=(x,x+y]$, where
\begin{equation}
t_{k}=\left\{
\begin{array}{ll}
2, \ \ &\textup{if} \ k=2,\\
k^{2}-k+1, \ \ &\textup{if} \ k\geq3.
\end{array}
\right.
\end{equation}
Then either
\begin{align}\label{k1}
\sum_{n\in A}e(\alpha n^{k}) \ll y^{1-\rho+\varepsilon},
\end{align}
or there exist integers $a$ and $q$ such that
\begin{align}\label{k2}
1\leq q\leq y^{k\rho},\ \ (a,q)=1, \ \ |q\alpha-a|\leq x^{1-k}y^{k\rho-1},
\end{align}
and
\begin{align}\label{k3}
\sum_{n\in A}e(\alpha n^{k}) \ll \frac{{\color{blue}w_{k}(q)}y}{1+yx^{k-1}|\alpha-a/q|}+y^{1-\rho+\varepsilon}.
\end{align}
\end{Lemma}

\begin{proof}
Similar as the proof in \cite{KL,H,D}, but estimate
\[
I(\beta)=\int_{x-y}^{x+y}e(\beta\gamma^{k})d\gamma
\]
in another way (as the way in \cite{K}), where $\beta=\alpha-a/q.$
By Lemma 6.2 in \cite{Va1}, we have
\[
I(\beta)\ll w_{k}(q)y(1+|\beta|yx^{k-1})^{-1}
\]
instead of the estimates in \cite{KL} and \cite{D}, i.e.
\[
I(\beta)\ll y(q+q|\beta|yx^{k-1})^{-1/k}.\]
Combining the analysis in \cite{KL,H}, we finally complete the proof.
\end{proof}

\begin{Lemma}[see \cite{MS}, see also in \cite{LL} for the case of $k=2$]\label{MS}
Let $y=x^{\theta}$ for some fixed $\theta\geq2/3+\varepsilon.$ Let $\alpha\in\mathbb{R}$ and let $k$ be a positive integer. Suppose that
\[
\left|\sum_{x<n\leq x+y}\Lambda(n)e(\alpha n^{k})  \right|\geq \frac{y}{(\log x)^{A}}
\]
for some $A\geq2.$ Then there exists a positive integer $q\leq (\log x)^{O_{k}(A)}$ such that
\[
|q\alpha-a|x^{k-1}y\leq (\log x)^{O_{k}(A)}.
\]
\end{Lemma}

\begin{Lemma}[see \cite{KL}]\label{K}
For $y\geq x^{1/2+\varepsilon},$  $s\geq k^{2}+k$ and $k\geq2,$ we have
\[
\int_{[0,1]}\left|\sum_{x<n\leq x+y}(\log p)e(\alpha p)\right|^{s}\mathrm{d}\alpha\ll y^{s-1}x^{1-k+\varepsilon}.
\]
\end{Lemma}

Next we will prove another lemma to achieve our goals. The proof involves some ideas in \cite{KZ,Yao,Zh}.
\begin{Lemma}\label{l1}
Suppose that $y=x^{\theta},$ $0<\rho\leq t_{k}^{-1},$ $\frac{1}{2-t_{k}\rho}\leq \theta\leq 1.$  For $\mathcal{I},$ a subinterval of $(x,x+y],$  let $h(\alpha)=h_{\mathcal{I}}(\alpha)=\sum_{n\in\mathcal{I}}(\log n) e(\alpha n^{k}).$ For $1\leq  q \leq y^{k\rho},$ $(a,q)=1,$ $\mathcal{M}_{q,a}=\{\alpha:|q\alpha-a|\leq x^{1-k}y^{k\rho-1}\},$ let $\mathcal{M}$ be the union of the intervals of $\mathcal{M}(q,a).$ Then for any measurable set $\mathfrak{W}\subseteq[0,1],$ we have
\[
\int_{\mathfrak{W}}|h(\alpha)|^{s}d\alpha\ll y^{2}(\log x)^{2}\mathcal{J}_{0}^{1/2}\left(\int_{\mathfrak{W}}
|h(\alpha)|^{2s-6}\right)^{1/2}
+y^{2-\rho+\varepsilon}\mathcal{J}(\mathfrak{W})\]
where
\[
\mathcal{J}_{0}=\sup_{\beta\in[0,1]}\int_{\mathcal{M}}\frac{w_{k}^{2}(q)|h(\alpha\pm\beta)|^{2} }{\left(1+x^{k-1}y|\alpha-a/q |\right)^{2}}d\alpha
\]
and
\[
\mathcal{J}(\mathfrak{W})=\int_{\mathfrak{W}}|h(\alpha)|^{{\bf s-2}}d\alpha.
\]
\end{Lemma}
\begin{proof}
Write
\[
\mathcal{M}_{\alpha}=\mathcal{M}_{\alpha}(q,a)-\beta=\{\beta:|q\alpha-q\beta-a|\leq
x^{1-k}y^{k\rho-1}\}
\]
and
\[
\mathcal{M}_{\alpha}=\bigcup_{q\leq y^{k\rho}}\bigcup_{1\leq a\leq q,\ (a,q)=1}\mathcal{M}_{\alpha}(q,a).
\]
Following the argument in \cite{Yao},
one can get
\begin{align*}
\begin{split}
\left|\int_{\mathfrak{W}}|h(\alpha)|^{s}d\alpha\right|^{2}&=\left|\sum_{n\in\mathcal{I}}
(\log n)\int_{\mathfrak{W}}e(\alpha n^{k})|h(\alpha)|^{s-2}h(-\alpha)d\alpha\right|^{2}\\
&\ll \left|\log x \sum_{x< n\leq x+y}
\left|\int_{\mathfrak{W}}e(\alpha n^{k})|h(\alpha)|^{s-2}h(-\alpha)d\alpha\right| \right|^{2}.
\end{split}
\end{align*}
By Cauchy's inequality, we have
\begin{align}\label{2.1}
\begin{split}
\left|\int_{\mathfrak{W}}|h(\alpha)|^{s}d\alpha\right|^{2}
\ll y(\log x)^{2} \sum_{x< n\leq x+y}
\left|\int_{\mathfrak{W}}e(\alpha n^{k})|h(\alpha)|^{s-2}h(-\alpha)d\alpha\right|^{2}.
\end{split}
\end{align}
Expand the square, we can get
\begin{align*}
\sum_{x< n\leq x+y}
&\left|\int_{\mathfrak{W}}e(\alpha n^{k})|h(\alpha)|^{s-2}h(-\alpha)d\alpha\right|^{2}\\
&=\int_{\mathfrak{W}}\int_{\mathfrak{W}}\sum_{x< n\leq x+y}e((\alpha-\beta) n^{k})|h(\alpha)|^{s-2}h(-\alpha)|h(\beta)|^{s-2}h(\beta)d\alpha d\beta
.\end{align*}
Write
\[
f(\gamma)=\sum_{x< n\leq x+y}e(\gamma n^{k}).
\]
Note that
\[
|h(-\alpha)h(\beta)|\leq |h(\alpha) |^{2}+| h(\beta)|^{2}.
\]
Then
\begin{align*}
\sum_{x< n\leq x+y}
&\left|\int_{\mathfrak{W}}e(\alpha n^{k})|h(\alpha)|^{s-2}h(-\alpha)d\alpha\right|^{2}\\&\ll
\int_{\mathfrak{W}}\int_{\mathfrak{W}}\left|f(\alpha-\beta)\right|\left(\left|h(\alpha)\right|^{s-2}| h(\beta)|^{s}  \right)  d\alpha d\beta.
\end{align*}
Thus
\begin{align*}
\int_{\mathfrak{W}}\int_{\mathfrak{W}}\left|f(\alpha-\beta)\right|\left|h(\alpha)\right| ^{s-2}| h(\beta)|^{s} d\alpha d\beta
&\ll
\int_{\mathfrak{W}}\int_{\mathfrak{W}\cap\mathcal{M}_{\beta}}|h(\beta)|^{s}|h(\alpha)|^{s-2}
\frac{w_{k}(q)y}{1+x^{k-1}y|\alpha-\beta-a/q|}d\alpha d\beta\\
&+ y^{1-\rho+\varepsilon}\int_{\mathfrak{W}}|h(\alpha)|^{s-2}
d\alpha\int_{\mathfrak{W}}|h(\alpha)|^{s}d\alpha.
\end{align*}
 Write
\[
T_{\beta}=\int_{\mathfrak{W}\cap\mathcal{M}_{\alpha}}|h(\beta)|^{s-2}
\frac{w_{k}(q)y}{1+x^{k-1}y|\alpha-\beta-a/q|}d\beta.
\]
By Cauchy's inequality, we have
\begin{align*}
 T_{\beta}&\ll
 y\left(\int_{\mathfrak{W}} |h(\beta)|^{2s-6}d\beta\right)^{1/2}\left(\int_{\mathcal{M}}\frac{w_{k}^{2}(q)|h(\alpha-\beta)|^{2} }{\left(1+x^{k-1}y|\beta-a/q|\right)^{2}}d\beta\right)^{1/2}\\
 &\ll
  y\left(\int_{\mathfrak{W}} |h(\beta)|^{2s-6}d\beta\right)^{1/2}\mathcal{J}_{0}^{1/2}.
\end{align*}
Hence
\begin{align}\label{2.2}
\begin{split}
\sum_{x< n\leq x+y}
\left|\int_{\mathfrak{W}}e(\alpha n^{k})|h(\alpha)|^{s-2}h(-\alpha)d\alpha\right|^{2}&\ll
y\int_{\mathfrak{W}}|h(\alpha)|^{s}d\alpha\left(\int_{\mathfrak{W}} |h(\alpha)|^{2s-6}d\alpha\right)^{1/2}\mathcal{J}_{0}^{1/2}\\
&+ y^{1-\rho+\varepsilon}\int_{\mathfrak{W}}|h(\alpha)|^{s-2}
d\alpha\int_{\mathfrak{W}}|h(\alpha)|^{s}d\alpha.
\end{split}
\end{align}
Now the desired result can be deduced from (\ref{2.1}) and (\ref{2.2}).

\end{proof}
\section{Proof of  Theorem \ref{t3}}
Let $x=(N/s)^{1/k}$ and $y=x^{\theta},$ where $N$ is a sufficiently large natural number. Let $n$ satisfy $ |n-N| \leq yx^{k-1}.$
Let
\[
f (\alpha)=\sum_{x -y <n\leq x +y }(\log p)e(\alpha p^{k}).
\]
denote the summation for primes $p.$
Then
\[
\varrho  (n)=\int_{[0,1]}f (\alpha)^{s}e(-n\alpha)d\alpha.
\]
Let
\[
P=(\log x)^{A}, \ \ \ \ Q= xy^{k-1}P^{-1}.
\]
Write
\[
I(q,a)=\{\alpha\in[0,1):|q\alpha-a|\leq Q^{-1}\}.
\]
Then define the major arcs $\mathfrak{M}$ and minor arcs $\mathfrak{m}$ as
\[
\mathfrak{M}=\mathfrak{M}(P,Q)=\bigcup_{q\leq P}\bigcup_{\mbox{\tiny$\begin{array}{c}
1\leq a\leq q\\
(a,q)=1\\
\end{array}$}}I(q,a),
\]
\[
\mathfrak{m}=\mathfrak{m}(P,Q)=[0,1)\setminus\mathfrak{M}.
\]
Suppose that $\mathfrak{B}$ is a measurable subset of $[0,1)$. Let
\[
\varrho (n,\mathfrak{B})=\int_{\mathfrak{B}}f(\alpha)^{s}e(-n\alpha)d\alpha.
\]
Then we can write
\[
\varrho  (n)=\varrho  (n,\mathfrak{M})+\varrho (n,\mathfrak{m}).
\]
The analysis of $\varrho (n,\mathfrak{M})$: We follow the analysis in \cite{MS} and \cite{WW}.
The width of our major arc is chosen so that if $\alpha\in \mathfrak{M}(q,a),$ then $f(\alpha)$ can be estimated by counting primes in short intervals in residue classes modulo $q.$ Since $\theta>7/12,$ we can use Huxley's result on primes in short intervals to get
\[
f(a/q+\beta)=\phi(q)^{-1}S(q,a)v(\beta)+
O(y/(\log)^{10})
\]
for
\[
|\beta|\leq P/(yx^{k-1}),
\]
where
\[
S(q,a)=\sum_{1\leq b\leq q \atop (b,q)=1}e\left(\frac{ab^{k}}{q}\right)
\]
and
\[
v(\beta)=k^{-1}\sum_{(x-y)^{k}\leq m \leq(x+y)^{k}}m^{-1+1/k}e(\beta m).
\]
Thus the standard theory of the major arc contributions in the Waring-Goldbach problem can be applied to yield the estimate
\begin{align}\label{asy}
\varrho  (n,\mathfrak{M})=\mathfrak{S}(n)\mathfrak{J}(n)
+O\left(\frac{y^{s-1}}{x^{k-1}(\log x)^{10}}\right),
\end{align}
where $\mathfrak{S}(n)$ is the singular series
\[
\mathfrak{S}(n)=\sum_{q=1}^{\infty}\phi(q)^{-s}
\sum_{1\leq a \leq q \atop (a,q)=1}S(q,a)^{s}e(-an/q),
\]
and $\mathfrak{J}(n)$ is the singular integral
\[
\mathfrak{J}(n)=\int_{[0,1]}v(\beta)^{s}e(-\beta n)d\beta.
\]
On the  other hand,
\[
\mathfrak{S}(n)\asymp1,\ \ \ \mathfrak{J}(n)\asymp y^{s-1}x^{1-k}.
\]
Then we have
\begin{equation}\label{z}
\varrho (n,\mathfrak{M}) \gg y^{s-1}x^{1-k}.
\end{equation}

$E(N,s,x^{k-1}y)$ does not exceed the number of integers $n$ for which
\[
\varrho(n,\mathfrak{m})\ll y^{s-1}x^{1-k}(\log x)^{-2}.
\]
Thus it follows from Bessel's inequality that
\[
E(N,s,x^{k-1}y)\ll (y^{s-1}x^{1-k})^{-2}(\log x)^{4}\int_{\mathfrak{m}}|f(\alpha)|^{2s}d \alpha,
\]
and theorem follows from the following proposition.
Then I would state and prove the proposition, and that will close the paper.
\begin{Proposition}\label{W1}
Let $A_{0}\geq 1$ be fixed and
\[
K(t)=\int_{\mathfrak{m}}|f(\alpha)|^{t}d\alpha.
\]
For $t\geq k^{2}+k+2$ and $y\geq x^{2/3 +\varepsilon},$ we have
\[
K(t)\ll y^{t-1}x^{1-k}(\log x)^{-A_{0}},
\]
provided that $A$ is sufficiently large in terms of  $A_{0}.$
\end{Proposition}

\begin{proof}
Fix $A_{0}.$
Choose $B=A_{0}+c_{0}+10.$  Choose $A$  sufficiently large in terms of $B.$ Then Lemma \ref{MS} implies that for $y\geq x^{2/3+\varepsilon},$
\begin{equation}\label{e}
\sup_{\alpha\in\mathfrak{m}}|f(\alpha)|\leq y (\log x)^{-B}.
\end{equation}

Choose $h(\alpha)=f(\alpha)$ and $\mathfrak{W}=\mathfrak{m}$ in Lemma \ref{l1}.
Then we have
\[
K(t)\ll y^{2} (\log x)^{2} \mathcal{J}_{0}^{1/2} K(2t-6)^{1/2}+
y^{2-\rho+\varepsilon}K(t-2).
\]
{\bf (i)} For the  first term, we have
\[K(t)\ll y^{2} (\log x)^{2} \mathcal{J}_{0}^{1/2} K(2t-6)^{1/2}
\ll y^{2} (\log x)^{2} \mathcal{J}_{0}^{1/2} K(t)^{1/2} \max_{\alpha\in [0,1]}f(\alpha)^{(t-6)/2}.
\]
Hence
\[K(t)^{2}
\ll y^{4} (\log x)^{4} \mathcal{J}_{0} K(t) \max_{\alpha\in [0,1]}f(\alpha)^{(t-6)}.
\]
Choosing $D=yx^{k-1}$ in Lemma \ref{Y}, this gives
\[K(t)
\ll y^{4} (\log x)^{4+c_{0}} (y^{2}(yx^{k-1})^{-1})\max_{\alpha\in [0,1]}f(\alpha)^{(t-6)}.
\]
Then by (\ref{e}), we have
\[K(t)
\ll y^{t-1} x^{1-k}(\log x)^{-A_{0}}.
\]
{\bf (ii)} For the  second term, by the Vinogradov mean value theorem (see Lemma \ref{K})
\[
\int_{[0,1)}f^{s}(\alpha) \ \mathrm{d}\alpha \ll y^{s-1}x^{1-k+\varepsilon},\ \ s\geq k^{2}+k,
\]
we have
\[K(t)
\ll y^{t-3} x^{1-k+\varepsilon}y^{2-\rho+\varepsilon}
\ll y^{t-1-\rho+\varepsilon}x^{1-k}.\]
Choosing $\varepsilon>0$ sufficiently small to ensure that $\rho=4\varepsilon$ is admissible in Lemma \ref{t2}. Then we can complete the proof.
\end{proof}

By the standard argument, we can get
\[
E(N,s,x^{k-1}y)\ll (y^{s-1}x^{1-k})^{-2}K(2s).
\]
Then Theorem \ref{t3} can be deduced by the standard argument, (\ref{z}) and Proposition \ref{W1}.

\bigskip
$\mathbf{Acknowledgements}$
I am deeply grateful to the referee(s) 
for carefully reading the manuscript 
and giving corrections, useful suggestions. 

\address{Wei Zhang\\ School of Mathematics and Statistics\\
               Henan University\\
               Kaifeng  475004, Henan\\
               China}
\email{zhangweimath@126.com}

\end{document}